\documentclass[a4paper]{article}

\pdfoutput=1

\usepackage{epic}
\usepackage{eepic}
\usepackage{amsmath}
\usepackage{amssymb}
\usepackage{amsthm}
\usepackage{enumerate}
\usepackage{mathtools}

\let\OLDthebibliography\thebibliography
\renewcommand\thebibliography[1]{
	\OLDthebibliography{#1}
	\setlength{\itemsep}{0pt}
}

\newcommand{\al}{\alpha}

\newcommand{\C}{\mathbb C}
\newcommand{\Q}{\mathbb Q}

\newcommand{\Z}{\mathbb Z}

\newcommand{\image}{\operatorname{im}}

\newcommand{\sign}{\operatorname{sign}}

\newcommand{\rk}{\operatorname{rk}}
\newcommand{\ord}{\operatorname{ord}}

\newcommand{\q}{\operatorname{q}}
\newcommand{\ka}{k}
\newcommand{\spann}{\operatorname{span}}

\newcommand{\leg}[2]{\left(\frac{#1}{#2}\right)}
\newcommand{\sumstack}[1]{\sum_{\substack{#1}}}
\newcommand{\SL}{\mathrm{SL}_2(\Z)}
\newcommand{\SLn}[1]{\mathrm{SL}_2(\Z/#1\Z)}
\newcommand{\Mp}{\mathrm{Mp}_2(\Z)}

\newcommand{\mainthm}{Let $D$ be a discriminant form. Then all modular forms for the Weil representation of $D$ are linear combinations of modular forms of the form $\uparrow_H\!(F)$, where $H\subset D$ is an isotropic subgroup such that $H^\bot/H$ is of small type and $F$ is a modular form for the Weil representation of $H^\bot/H$. For any discriminant form of small type there exist modular forms which are not linear combinations of isotropically lifted modular forms.}

\newtheorem{thm}{Theorem}[section]
\newtheorem{prp}[thm]{Proposition}
\newtheorem{cor}[thm]{Corollary}
\newtheorem{lem}[thm]{Lemma}

\theoremstyle{definition}

\theoremstyle{remark}

\begin{document}

\begin{center}

{\Large\bf Modular forms for the Weil representation\\[2mm]
induced from isotropic subgroups}\\[12mm]
Manuel K.-H. M\"uller,\\
Fachbereich Mathematik, Technische Universit\"at Darmstadt, 
Schlo{\ss}gartenstra{\ss}e 7, 64289 Darmstadt, Deutschland\\
mmueller@mathematik.tu-darmstadt.de\end{center}

\vspace*{2cm}
\noindent
For an isotropic subgroup $H$ of a discriminant form $D$ there exists a lift from modular forms for the Weil representation of the discriminant form $H^\bot/H$ to modular forms for the Weil representation of $D$. We determine a set of discriminant forms such that all modular forms for any discriminant form are induced from the discriminant forms in this set. Furthermore for any discriminant form in this set there exist modular forms that are not induced from smaller discriminant forms.

\vspace*{0.25cm}
\noindent
Keywords: Weil representation, modular forms of half-integer weight, vector valued modular forms


\vspace*{1.25cm}
\begin{tabular}{rl}
1  & Introduction \\
2  & Definitions \\
3  & Reduction to lifting pointwise \\
4  & Image of the  lifts \\
5  & Main theorem \\
6  & Acknowledgements
\end{tabular}
\vspace*{0.8cm}

\section{Introduction}

The metaplectic group $\Mp$ acts on the group ring $\C[D]$ of a discriminant form $D$. This representation is called the Weil representation. It is a special case of representations of $\mathrm{Sp}_{2n}$ constructed in \cite{W}, which play an important role in the theory of automorphic forms.

Modular forms for the Weil representation are often induced from modular forms for smaller discriminant forms. Since these are much easier to understand, it is important to know when this happens. In the present paper we describe for which discriminant forms all modular forms for the Weil representation are linear combinations of lifts and for which discriminant forms there exist modular forms which are not induced from smaller discriminant forms. If one restricts to modular forms of a specific weight $k\in\frac{1}{2}\Z$, then the list of discriminant forms for which modular forms not induced from smaller discriminant forms exist gets smaller and it is in general much more difficult to determine this smaller list. For the case of weight $0$, i.e. the invariants of the Weil representation, it was shown in \cite{MS} that all invariants are induced from $5$ fundamental discriminant forms. 

We describe our results in more detail. A discriminant form is a finite abelian group $D$ with a non-degenerate quadratic form $\q : D \to \Q/\Z$. The level of $D$ is the smallest positive integer $N$ such that $N\q(\gamma) \in \Z$ for all $\gamma \in D$. Every discriminant form can be realized as the dual quotient $L'/L$ of an even lattice $L$. The signature of $L$ is unique modulo $8$. We can even assume that $L$ is positive definite. The vector valued theta function $\theta_L$ of $L$ takes values on the group algebra $\C[D]$, generated by the formal basis $(e^\gamma)_{\gamma\in D}$. It is given by
\[ \theta(\tau) = 
\sum_{\gamma \in D} \theta_{\gamma}(\tau) e^{\gamma} \]
with $\theta_{\gamma}(\tau) = \sum_{\al \in \gamma+L} q^{\al^2/2}$.
The Poisson summation formula implies that $\theta$ transforms as a vector valued modular form under the metaplectic cover $\Mp$ of $\SL$. The corresponding representation $\rho_D$ of $\Mp$ on the group algebra $\C[D]$ is called the Weil representation of $\Mp$. We say that a function $F:\mathbb{H}\to\C[D]$ is a modular form of weight $\ka\in\frac{1}{2}\Z$ for the Weil representation if $F$ is holomorphic, holomorphic at the cusps and satisfies
\begin{align*}
    F(M\tau) = \phi(\tau)^{2\ka}\rho_D(M,\phi)F(\tau)
\end{align*}
for all $(M,\phi)\in\Mp$, where $M\tau = \frac{a\tau+b}{c\tau+d}$ denotes the M\"obius transformation of $M=\left(\begin{smallmatrix}
	a&b\\c&d
\end{smallmatrix}\right)\in\mathrm{SL}_2(\Z)$. We denote the space of modular forms for the Weil representation of $D$ by $\mathrm{M}_\ka(D)$.

Let $H$ be an isotropic subgroup of $D$. Then $H^{\perp}/H$ is a discriminant form of the same signature as $D$ and of order $|H^{\perp}/H| = |D|/|H|^2$. There is an isotropic lift $\uparrow_H : \C[H^{\perp}/H] \to \C[D]$ which commutes with the corresponding Weil representations (see Section \ref{sec:Definitions}). In particular $\uparrow_H$ maps modular forms to modular forms. Our goal is to determine a list of discriminant forms such that modular forms of the form $\uparrow_H(F)$, where $H^\bot/H$ is isomorphic to a discriminant form in said list generate all modular forms for $D$. For this purpose we introduce the notion of small type:\\
First assume that $D$ is a discriminant form of level a power of a prime $p$. If $p$ is odd we say that $D$ is of small type if one of the following conditions holds:
\begin{enumerate}[(i)]
	\item $D$ has rank two or less.
	\item $D$ has rank three and at least one Jordan component is of level $p$.
	\item $D$ has rank four and is of type $p^{-\epsilon2}q_1^{\pm1}q_2^{\pm1}$, where $\epsilon = \leg{-1}{p}$ and $q_1,q_2$ are powers of $p$ and can also be $p$.
	\item $D$ has rank five and is of level $p$.
\end{enumerate}
If $p=2$ there is a similar characterization (cf. Section \ref{sec:MainTheorem}). If $D$ has level $N$ and $N=\prod_{p|N} p^{\nu_p}$ is the prime decomposition of $N$, then $D$ decomposes into the orthogonal sum of its $p$-subgroups 
\[  D = \bigoplus_{p|N} D_{p^{\nu_p}} \,   \]
where $D_{p^{\nu_p}}$ is the kernel of $\gamma\mapsto p^{\nu_p}\cdot\gamma$. We say that a discriminant form of level $N$ is of small type if for all $p\mid N$ the $p$-subgroups $D_{p^{\nu_p}}$ of $D$ are of small type. We remark that any discriminant form of rank $\geq7$ is not of small type. Now the main result of the present paper is Theorem \ref{thm:mainthm}:

\medskip

{\em \mainthm}

\medskip

This result is one ingredient for a theory of vector valued new forms for the Weil representation, which is still to be developed.

We remark that in his Ph.D.\ Thesis \cite{Wr} Werner showed that for a discriminant form $D$ of level $N$ all modular forms are linear combinations of isotropically lifted modular forms if $|D|\geq N^9$. In contrast to Werner's result Theorem \ref{thm:mainthm} is sharp.

We sketch the proof of the theorem. We say that a modular form $F\in\mathrm{M}_\ka(D)$ is a linear combination of isotropically lifted modular forms if we can write
\[ F = \sum_{0\not= H}\uparrow_H(F_H) \]
for suitable modular forms $F_H\in\mathrm{M}_\ka(H^\bot/H)$. Here we sum over all non-trivial isotropic subgroups. We show that for a discriminant form $D$ all modular forms are linear combinations of isotropically lifted modular forms if and only if $D$ is not of small type. Using the fact that the lifts are transitive, Theorem \ref{thm:mainthm} then follows by induction.

First we show that for a modular form being a linear combination of isotropically lifted modular forms is actually a pointwise property: $F$ is a linear combination of isotropically lifted modular forms if and only if for every $\tau$ the point $F(\tau)\in\C[D]$ is a linear combination of isotropically lifted points $v\in\C[H^\bot/H]$. One direction of this equality is trivial, the other one is proved in Proposition \ref{prp:iTildeIsI}. 
For $\ka$ large enough every subrepresentation of $\rho_D$ contains some non-trivial modular form of weight $\ka$. We deduce that all modular forms for all weights are linear combinations of isotropically lifted modular forms if and only if the space generated by the images of the maps $\uparrow_H$ is all of $\C[D]$ (see Theorem \ref{thm:MainThmNeedsAllCD}). The latter question can be reduced to the $p$-subgroups.

So it remains to show that the isotropic lifts generate $\C[D]$ if and only if $D$ is not of small type. The idea of the proof is to find some condition on $\gamma\in D$ which is equivalent to $e^\gamma$ being a linear combination of isotropic lifts. For $p$ odd this condition says that $\gamma^\bot$ contains some isotropic subgroup isomorphic to $(\Z/p\Z)^2$ (see Proposition \ref{prp:eGammaIff}). When $D$ contains an isotropic subgroup isomorphic to $(\Z/p\Z)^3$, then this is the case for all $\gamma\in D$. Interestingly the other direction holds as well for all discriminant forms except for $p^{-\epsilon6}$ with $\epsilon = \leg{-1}{p}$. This form does not contain an isotropic subgroup isomorphic to $(\Z/p\Z)^3$, but still for every $\gamma\in D$ the subgroup $\gamma^\bot$ contains an isotropic subgroup isomorphic to $(\Z/p\Z)^2$. The form $p^{-\epsilon6}$ is the only discriminant form with this property.

In the case $p=2$ the same condition also implies that $e^\gamma\in\image(\uparrow)$, however it is too strong. We will use a result from graph theory for a sharper condition (see Proposition \ref{prp:eGammaIff21}).

In both cases we determine the discriminant forms that do not contain an isotropic subgroup isomorphic to $(\Z/p\Z)^3$ and then check for which of them the condition holds for all $\gamma\in D$ and for which there exists a $\gamma\in D$, where the condition does not hold (see Theorem \ref{thm:mainp} and \ref{thm:main2}). We remark that, if a $p$-adic discriminant form is of small type, then it does not contain an isotropic subgroup isomorphic to $(\Z/p\Z)^3$. Furthermore, there are very few $p$-adic discriminant forms that are not of small type and do not contain an isotropic subgroup isomorphic to $(\Z/p\Z)^3$.


\section{Definitions}\label{sec:Definitions}

In this section we recall some results on discriminant forms, the Weil representation and modular forms.  References are \cite {AGM}, \cite{Bo2}, \cite{CS}, \cite{N}, \cite{S1} and \cite{Sk}.

\medskip


A discriminant form is a finite abelian group $D$ with a quadratic form $\q : D \to \Q/\Z$ such that $(\beta,\gamma) = \q(\beta + \gamma)- \q(\beta) - \q(\gamma) \bmod 1$ is a non-degenerate symmetric bilinear form. The level of $D$ is the smallest positive integer $N$ such that
$N\q(\gamma) = 0 \bmod 1$ for all $\gamma \in D$. An element $\gamma\in D$ is called isotropic if $\q(\gamma) = 0\bmod1$ and anisotropic otherwise.



If $L$ is an even lattice then $L'/L$ is a discriminant form with the quadratic form given by $\q(\gamma) = \gamma^2/2 \bmod 1$. Conversely, every discriminant form can be obtained in this way. The corresponding lattice can be chosen to be positive definite. The signature $\sign(D) \in \Z/8\Z$ of a discriminant form $D$ is defined as the signature modulo $8$ of any even lattice with that discriminant form.

 
Every discriminant form decomposes into a sum of Jordan components and every Jordan component can be written as a sum of indecomposable Jordan components (usually not uniquely). For the Jordan components we use the notation introduced by Conway and Sloane (cf.\ \cite{CS}, chapter 15). Recall that for $q$ a power of an odd prime $p$ the symbol $q^{\pm n}$ denotes a discriminant form that, as a group, is isomorphic to $(\Z/q\Z)^n$ and has level $q$. For $q$ a power of $2$ the symbol $q_{I\!I}^{\pm 2n}$ denotes a discriminant form that, as a group, is isomorphic to $(\Z/q\Z)^{2n}$ and has level $q$ and $q_{t}^{\pm n}$ denotes a discriminant form that, as a group, is isomorphic to $(\Z/q\Z)^{n}$ and has level $2q$. The former are called the even $2$-adic components, the latter the odd $2$-adic components.

\medskip

The metaplectic group $\Mp$ is a double cover of $\SL$. It can be described as the group of pairs $\left(\left(\begin{smallmatrix} a & b \\ c & d \end{smallmatrix}\right),\phi(\tau)\right)$, where $\left(\begin{smallmatrix} a & b \\ c & d \end{smallmatrix}\right)\in\SL$ and $\phi$ is a holomorphic function on the upper half plane such that $\phi(\tau)^2 = c\tau+d$. The product of $(M_1, \phi_1(\tau)), \ (M_2, \phi_2(\tau)) \in \Mp$ is given by
\begin{align*}
	(M_1,\phi_1(\tau))(M_2,\phi_2(\tau)) = (M_1M_2,\phi_1(M_2\tau)\phi_2(\tau)),
\end{align*}
where $M\tau = \frac{a\tau+b}{c\tau+d}$ denotes the usual M\"obius transform on $\mathbb{H}$. The standard generators of $\Mp$ are
$S = \left(\left( \begin{smallmatrix} 0 & -1 \\ 1 & 0 \end{smallmatrix} \right),\sqrt{\tau}\right)$ and 
$T = \left(\left( \begin{smallmatrix} 1 &  1 \\ 0 & 1 \end{smallmatrix} \right),1\right)$. Let $D$ be a discriminant form of level $N$ and let $\C [D]$ be its group ring generated by a formal basis $(e^\gamma)_{\gamma\in D}$. Then the \textit{Weil representation} of $\Mp$ is defined as (cf. \cite{Bo1})
\begin{align*} 
\rho_D(T) e^{\gamma}  & = e(\q(\gamma))\, e^{\gamma} \\
\rho_D(S) e^{\gamma}  & = \frac{e(-\sign(D)/8)}{\sqrt{|D|}}
                  \sum_{\beta\in D} e(-(\gamma,\beta))\, e^{\beta},
\end{align*}
where $e(z) := e^{2\pi i z}$. For discriminant forms of even signature the inverse image of the group $\Gamma(N) = \{M\in\SL\mid M=I\bmod N\}$ under the covering map $\Mp\rightarrow\SL$ acts trivially in the Weil representation. For odd signature there is a unique section $s$ such that $s(\Gamma(N))$ acts trivially. We denote the corresponding group in both cases by $\mathrm{Mp}_2(N)$. The quotient $\Mp/\mathrm{Mp}_2(N)$ is isomorphic to $\SLn{N}$ for even signature and to a double cover of $\SLn{N}$ for odd signature.

We define a scalar product on the group ring $\C[D]$ which is linear in the first and antilinear in the second variable by
\begin{align*}
	\langle e^{\gamma},e^{\beta}\rangle= \delta_{\gamma \beta}.
\end{align*}
The Weil representation is unitary with respect to this scalar product.

Let $\ka\in\frac{1}{2}\Z$ and $f$ be a function from $\mathbb{H}$ to a complex vector space. For $(M,\phi)\in\Mp$ we define the Petersson-slash operator $|_\ka$ by
\begin{equation*}
    (f|_\ka(M,\phi))(\tau) = \phi(\tau)^{-2\ka}f(M\tau).
\end{equation*}
A function $F:\mathbb{H} \to \C[D]$ is called a modular form of weight $\ka$ with respect to $\rho_D$ and $\Mp$ if
\begin{enumerate}[(i)]
	\item $F|_\ka(M,\phi) = \rho_D(M,\phi)F$ for all $(M,\phi)\in\Mp$,
	\item $F$ is holomorphic on $\mathbb{H}$,
	\item $F$ is holomorphic at the cusp $\infty$.
\end{enumerate}
Here condition (iii) means that $F$ has a Fourier expansion of the form
\begin{equation*}
    F(\tau) = \sum_{\gamma\in D}\sumstack{n\in\Z+\q(\gamma) \\ n\geq0}^\infty c(\gamma,n)e(n\tau)e^\gamma.
\end{equation*}
Moreover, if all $c(\gamma, n)$ with $n = 0$ vanish, then $F$ is called a cusp form. The
$\C$-vector space of modular forms of weight $\ka$ with respect to $\rho_D$ and $\Mp$ is denoted by $\mathrm{M}_\ka(D)$, the subspace of cusp forms by $\mathrm{S}_\ka(D)$.

We furthermore denote by $\mathrm{M}_\ka(\mathrm{Mp}_2(N))$ the space of scalar-valued modular forms for the subgroup $\mathrm{Mp}_2(N)$. Let $f\in\mathrm{M}_\ka(\mathrm{Mp}_2(N))$ and $v\in\C[D]$. Then
\begin{equation*}
    F_{f,v} := \frac{1}{|\mathrm{Mp}_2(N)\backslash\Mp|}\sumstack{(M,\phi)\in \\ \mathrm{Mp}_2(N)\backslash\Mp}f|_\ka(M,\phi)\rho_D(M,\phi)^{-1}v\in\mathrm{M}_\ka(D)
\end{equation*}
is a modular form for the Weil representation. These modular forms span $\mathrm{M}_\ka(D)$, when $f$ ranges over $\mathrm{M}_\ka(\mathrm{Mp}_2(N))$ and $v$ ranges over $\C[D]$ (This is a classical idea, cf. e.g. \cite{Se}, in the case of the Weil representation of $\SL$ see \cite{S2}). 

\medskip

Let $D$ be a discriminant form and $H$ an isotropic subgroup of $D$. Then $H^\bot/H$ is a discriminant form of the same signature as $D$ and order $|H^\bot/H| = |D|/|H|^2.$ There is an \textit{isotropic lift}
\[ \uparrow_H \ := \ \uparrow_H^D \ : \C[H^{\perp}/H] \to \C[D]  \]
defined by
\[  \uparrow_H(e^{\gamma+H}) = \sum_{\mu \in H} e^{\gamma + \mu}   \]
for $\gamma\in H^{\bot}$ and an \textit{isotropic descent} 
\[ \downarrow_H \ := \ \downarrow_H^D \ : \C[D] \to \C[H^{\bot}/H] \]
defined by
\[ \downarrow_H(e^{\gamma}) = 
\begin{cases}
\, e^{\gamma+H} & \text{if $\gamma\in H^{\perp}$}, \\
\,          0 & \text{otherwise} 
\end{cases} 
\]
(cf. e.g. \cite{Br}, \cite{S3}). The following results are easy to prove (see \cite{MS}).

\begin{prp}
Let $D$ be a discriminant form of even signature and $H$ an isotropic subgroup of $D$. Then the maps $\uparrow_H$ and $\downarrow_H$ are adjoint with respect to the inner products on $\C[H^{\perp}/H]$ and $\C[D]$ and commute with the Weil representations $\rho_{H^{\perp}/H}$ and $\rho_D$. In particular they map modular forms to modular forms. They also commute with $F_{f,v}$ as maps in $v$ on $\C[H^{\perp}/H]$ and $\C[D]$ respectively.
\end{prp}

The isotropic lift is transitive.

\begin{prp}\label{prp:liftIsTransitive}
Let $D$ be a discriminant form of even signature and $H \subset K$ isotropic subgroups of $D$. Then 
$H \subset K \subset K^{\perp} \subset H^{\perp}$ and $K/H$ is an isotropic subgroup of $H^{\perp}/H$ with orthogonal complement $K^{\perp}/H$. Moreover
\begin{align*}
    \uparrow_H^D \, \circ \, \uparrow_{K/H}^{H^{\perp}/H} \quad &= \quad \uparrow_K^D \text{ and} \\
    \downarrow_{K/H}^{H^{\perp}/H} \, \circ \, \downarrow_H^D \quad &= \quad \downarrow_K^D.
\end{align*}
\end{prp}

Let $\mathcal{H}$ be the set of all non-trivial isotropic subgroups $H$ of $D$. To simplify the notation we introduce the following subsets of $\C[D]$:
\begin{align*}
	\image(\uparrow) &= \sum_{H\in\mathcal{H}}\image(\uparrow_H)\quad \text{and} \\
	\ker(\downarrow) &= \bigcap_{H\in\mathcal{H}}\ker(\downarrow_H).
\end{align*}
Because $\uparrow_H$ and $\downarrow_H$ are adjoint we find that
\begin{align}\label{eq:imageIBot}
    \begin{split}
    \image(\uparrow)^\bot &= \ker(\downarrow)\quad \text{and} \\
    \image(\uparrow) &= \ker(\downarrow)^\bot
    \end{split}
\end{align}
with respect to the inner product on $\C[D]$.

\section{Reduction to lifting pointwise}
In this section we show that a modular form $F$ is a linear combination of isotropically lifted modular forms if and only if for every $\tau$ on the upper half plane $F(\tau)$ is a linear combination of lifts. In Theorem \ref{thm:MainThmNeedsAllCD} we will argue further that for a discriminant form $D$ all modular forms are linear combinations of isotropically lifted modular forms if and only if all of $\C[D]$ is in the image of the lifts.
\begin{prp}\label{prp:Fdecomposes}
Let $F\in\mathrm{M}_\ka(D)$ and let $V = \spann_{\tau\in\mathbb{H}}(F(\tau))\subset \C[D]$ and $(v_i)_{i\in I}$ an orthonormal basis of $V$ with respect to $\langle\cdot,\cdot\rangle$. Then there exist $(f_i)_{i\in I}$ with $f_i\in\mathrm{M}_\ka(\mathrm{Mp}_2(N))$ such that
\begin{equation*}
    F = \sum_{i\in I}f_iv_i.
\end{equation*}
In fact $f_i(\tau) = \langle F(\tau),v_i\rangle$.
\end{prp}
\begin{proof}
Since for all $\tau\in\mathbb{H}$ the point $F(\tau)$ is in $V$ and $(v_i)_{i\in I}$ is a orthonormal basis of $V$, the equality
\begin{equation*}
    \sum_{i\in I}\langle F(\tau),v_i\rangle v_i = F(\tau)
\end{equation*}
holds pointwise for all $\tau\in\mathbb{H}$. Hence, we only need to show that $f_i:=\langle F,v_i\rangle\in\mathrm{M}_\ka(\mathrm{Mp}_2(N))$: For any $(M,\phi)\in\Mp$ we have
\begin{align}\label{eq:rho_DSlash}
\begin{split}
    f_i|_\ka(M,\phi) &= \langle F(\tau)|_\ka(M,\phi),v_i\rangle \\
    &= \langle \rho_D(M,\phi)F(\tau),v_i\rangle \\
    &= \langle F(\tau),\rho_D(M,\phi)^{-1}v_i\rangle,
\end{split}
\end{align}
so if $(M,\phi)\in\mathrm{Mp}_2(N)$, then $\rho_D(M,\phi)^{-1}$ acts trivially on $v_i$ and $f_i$ is invariant under $\mathrm{Mp}_2(N)$. The holomorphicity of $F$ on $\mathbb{H}$ and at $\infty$ together with (\ref{eq:rho_DSlash}) implies holomorphicity of $f_i$ on $\mathbb{H}$ and at the cusps.
\end{proof}
Now we can show that we can restrict to pointwise lifting.
\begin{prp}\label{prp:iTildeIsI}
Let $D$ be a discriminant form and let $F\in \mathrm{M}_\ka(D)$. Then $F$ is a linear combination of isotropically lifted modular forms if and only if $F(\tau)\in\image(\uparrow)$ for all $\tau\in\mathbb{H}$.
\end{prp}
\begin{proof}
If $F$ is a linear combination of isotropically lifted modular forms, then it is clear that $F(\tau)\in\image(\uparrow)$. So let us assume that $F(\tau)\in\image(\uparrow)$ holds for all $\tau\in\mathbb{H}$. By Proposition \ref{prp:Fdecomposes} we can write
\begin{equation}\label{eq:FInBasis}
    F = \sum_{i\in I}f_iv_i,
\end{equation}
where $f_i\in\mathrm{M}_\ka(\mathrm{Mp}_2(N))$ and $(v_i)_{i\in I}$ is a basis of $V = \spann_{\tau\in\mathbb{H}}(F(\tau))$. But by assumption $V\subset\image(\uparrow)$. So we can write
\begin{equation}\label{eq:vAsLifts}
    v_i = \sum_{H\in\mathcal{H}}\uparrow_H(v_{i,H})
\end{equation}
for suitable $v_{i,H}\in\C[H^\bot/H]$. Since $F$ is a modular form for $\rho_D$ we have
\begin{align*}
    F &= \frac{1}{|\mathrm{Mp}_2(N)\backslash\Mp|}\sum_{(M,\phi)\in\mathrm{Mp}_2(N)\backslash\Mp}\rho_D(M,\phi)^{-1}F|_\ka(M,\phi).
\end{align*}
For better readability we denote $A := |\mathrm{Mp}_2(N)\backslash\Mp|$. Applying (\ref{eq:FInBasis}) and (\ref{eq:vAsLifts}) we get
\begin{align*}
    F &= A^{-1}\cdot\sum_{(M,\phi)}\sum_{i\in I}f_i|_\ka(M,\phi)\rho_D(M,\phi)^{-1}v_i \\
    &= A^{-1}\cdot\sum_{(M,\phi)}\sum_{i\in I}f_i|_\ka(M,\phi)\sum_{H\in\mathcal{H}}\rho_D(M,\phi)^{-1}\uparrow_H(v_{i,H}) \\
    &= \sum_{H\in\mathcal{H}}\uparrow_H\left(A^{-1}\cdot\sum_{(M,\phi)}\sum_{i\in I}f_i|_\ka(M,\phi)\rho_{H^\bot/H}(M,\phi)^{-1}v_{i,H}\right),
\end{align*}
where
\begin{align*}
    A^{-1}\cdot\sum_{(M,\phi)}\sum_{i\in I}f_i|_\ka(M,\phi)\rho_{H^\bot/H}(M,\phi)^{-1}v_{i,H} = \sum_{i\in I}F_{f_i,v_{i,H}}
\end{align*}
is a modular form for $\rho_{H^\bot/H}$.
\end{proof}
We remark that \cite[Theorem 115]{Wr} is equivalent to Proposition \ref{prp:iTildeIsI}, however our argument is shorter.

Clearly if $\image(\uparrow) = \C[D]$ then all $F\in\mathrm{M}_\ka(D)$ are linear combinations of isotropic lifts. We want to argue that the other direction holds as well.
\begin{prp}\label{prp:allVectorsAreNeeded}
Let $D$ be a discriminant form. Then
\begin{equation*}
    \spann\{F(\tau)\mid F\in\mathrm{M}_\ka(D),\ \tau\in\mathbb{H},\ \ka\in\textstyle{\frac{1}{2}}\Z\} = \C[D].
\end{equation*}
\end{prp}
\begin{proof}
Define $V:=\spann\{F(\tau)\mid F\in\mathrm{M}_\ka(D),\ \tau\in\mathbb{H},\ \ka\in\frac{1}{2}\Z\}$ and suppose that $V\subsetneq\C[D]$. Then $W:=V^\bot$ is non-trivial. Since $V$ is invariant under $\rho_D$ and the Weil representation is a unitary action for the scalar product, also $W$ is invariant under $\rho_D$. Therefore, $(W,\rho_D)$ is a representation of $\Mp$ with $W\not=\{0\}$. For $\ka$ large enough there exists a non-zero modular form $F$ for $(W,\rho_D)$ (This was shown in \cite[section 3]{KM} for the integral weight case, the general case is treated in \cite[Proposition 3.3]{CFK}). But then also $F\in\mathrm{M}_\ka(D)$ and $F(\tau)\in W$ for all $\tau\in\mathbb{H}$, which contradicts the definition of $W$.
\end{proof}
We get
\begin{thm}\label{thm:MainThmNeedsAllCD}
Let $D$ be a discriminant form. Then all modular forms for the Weil representation 
are linear combinations of isotropically lifted modular forms if and only if $\image(\uparrow) = \C[D]$.
\end{thm}
\begin{proof}
By Proposition \ref{prp:iTildeIsI}, $\image(\uparrow)=\C[D]$ implies that all modular forms for the Weil representation are linear combinations of isotropically lifted modular forms. If $\image(\uparrow)\subsetneq\C[D]$ then by Proposition \ref{prp:allVectorsAreNeeded} there exists some non-trivial modular form on $\image(\uparrow)^\bot$. By Proposition \ref{prp:iTildeIsI} this can not be a linear combination of isotropically lifted modular forms.
\end{proof}

We remark that the line of reasoning in this section can also be applied to other finite groups and representations (Recall that the Weil representation of $\SL$ descends to a representation of the finite group $\mathrm{Mp}_2(N)\backslash\Mp$).

\section{Image of the lifts}
In this section we want to investigate when $\image(\uparrow)$ contains all of $\C[D]$. We will show under what conditions we can write $e^\gamma$ as a linear combination of lifts for any $\gamma\in D$. Since $D$ and also any isotropic subgroup of $D$ decomposes into its $p$-subgroups we can restrict ourselves to the case where the level of $D$ is a power of a prime $p$. Therefore, in this section $p$ is some fixed prime and $D$ will always be a discriminant form of level a power of $p$.

The following lemma combines Lemmas 120 and 121 in \cite{Wr} and was also proved in \cite[Proposition 7.1]{MS}.
\begin{lem}\label{lem:eGammaInImage}
Let $\gamma \in D$. Suppose that $\gamma^{\bot}$ contains an isotropic subgroup $H$ isomorphic to $(\Z/p\Z)^2$. Then $e^{\gamma}\in\image(\uparrow)$.
\end{lem}
A simple condition on $D$ that ensures that the condition of Lemma \ref{lem:eGammaInImage} is satisfied for all $\gamma\in D$ is given in
\begin{lem}\label{lem:eGammaInImageHoch3}
If there exists an isotropic subgroup $H\subset D$ isomorphic to $(\Z/p\Z)^3$ then for all $\gamma\in D$ the group $\gamma^\bot$ contains an isotropic subgroup $H\subset D$ isomorphic to $(\Z/p\Z)^2$.
\end{lem}
\begin{proof}
Let $\gamma\in D$ and $H\subset D$ be an isotropic subgroup isomorphic to $(\Z/p\Z)^3$. If $\gamma\in H^\bot$, take any subgroup of $H$ isomorphic to $(\Z/p\Z)^2$. If $\gamma\not\in H^\bot$, we can choose generators $\alpha$, $\beta$, $\mu$ of $H$ such that $(\gamma,\mu)\not=0\bmod1$. Since $\alpha$, $\beta$ and $\mu$ are of order $p$ it follows that $(\alpha,\gamma)$, $(\beta,\gamma)$ and $(\mu,\gamma)$ take values in $\frac{1}{p}\Z$. This implies that we can choose $x,y\in\Z$ such that
\begin{align*}
    (\alpha + x\mu,\gamma) = (\alpha,\gamma) + x(\mu,\gamma) &= 0\mod1\text{ and} \\
    (\beta + y\mu,\gamma) = (\beta,\gamma) + y(\mu,\gamma) &= 0\mod1.
\end{align*}
Then $\langle\alpha + x\mu, \beta + y\mu\rangle$ is the desired group.
\end{proof}
\begin{cor}\label{cor:global}
If there exists an isotropic subgroup $H\subset D$ isomorphic to $(\Z/p\Z)^3$ then $\image(\uparrow) = \C[D]$.
\end{cor}
We will also need to understand when $e^\gamma\not\in\image(\uparrow)$. Recall that $\mathcal{H}$ denotes the set of non-trivial isotropic subgroups. We define $\mathcal{H}' = \{H\in\mathcal{H}\mid |H| = p\}$ and
\begin{align*}
    \image(\uparrow') &:= \sum_{H\in\mathcal{H}'}\image(\uparrow_H) \\
    \ker(\downarrow') &:= \bigcap_{H\in\mathcal{H}'}\ker(\downarrow_H).
\end{align*}
Similar to equation (\ref{eq:imageIBot}) also $\image(\uparrow') = \ker(\downarrow')^\bot$. For the construction of $e^\gamma$ as a linear combination of isotropic lifts in the proof of Lemma \ref{lem:eGammaInImage}, only subgroups of order $p$ are used. It turns out that this always suffices:
\begin{lem}\label{lem:iPrimeIsI}
One has $\image(\uparrow') = \image(\uparrow)$ and $\ker(\downarrow') = \ker(\downarrow)$.
\end{lem}
\begin{proof}
First we show that $\ker(\downarrow') = \ker(\downarrow)$. The direction $\ker(\downarrow') \supset \ker(\downarrow)$ is clear, so let $v\in\ker(\downarrow')$ and let $H$ be any non-trivial isotropic subgroup. Then $H$ contains an isotropic subgroup of order $p$ say $H'$ and $H/H'$ is an isotropic subgroup in $H^{\prime\bot}/H'$. Therefore
\begin{equation*}
    \downarrow_H^D(v) = \downarrow_{H/H'}^{H^{\prime\bot}/H'}(\downarrow_{H'}^D(v)) = \downarrow_{H/H'}^{H^{\prime\bot}/H'}(0) = 0
\end{equation*}
and $v\in\ker(\downarrow)$. Because $\image(\uparrow) = \ker(\downarrow)^\bot = \ker(\downarrow')^\bot = \image(\uparrow')$, we automatically get the other equality.
\end{proof}
We will later see that in some cases the condition in Lemma \ref{lem:eGammaInImage} is necessary for $e^\gamma$ to be in $\image(\uparrow)$. For this we will need
\begin{lem}\label{lem:elementInKernel}
Let $\gamma\in D$. Assume that $\gamma^\bot$ contains no isotropic subgroup $H\subset D$ isomorphic to $(\Z/p\Z)^2$. Then for
\begin{equation*}
    v := e^\gamma - \frac{1}{p-1}\sumstack{H'\in\mathcal{H}' \\ H'\subset\gamma^\bot}\sum_{\mu\in H'\setminus\{0\}}e^{\gamma+\mu}
\end{equation*}
we have $\downarrow_H(v) = 0$ for all non-trivial isotropic subgroups $H\subset\gamma^\bot$ and $\langle v,e^\gamma\rangle = 1$.
\end{lem}
\begin{proof}
The assertion $\langle v,e^\gamma\rangle = 1$ is clear from the construction. Let $H$ be any non-trivial isotropic subgroup with $H\subset\gamma^\bot$. Because of the assumption, $H$ contains exactly one subgroup of order $p$ say $H'$. The subgroup $H'$ appears in the sum defining $v$. For any $K\in\mathcal{H}'$ with $K\subset\gamma^\bot$ and $\mu\in K\setminus\{0\}$ we get the following equivalence:
\begin{align*}
    \gamma+\mu\in H^{\bot} \Leftrightarrow \mu\in H^{\bot}\Leftrightarrow \mu\in H'\Leftrightarrow K= H',
\end{align*}
because otherwise $\langle \mu,H'\rangle$ would contradict the assumption of the lemma. So
\begin{align*}
    \downarrow_H(v) &= \downarrow_H(e^\gamma) - \frac{1}{p-1}\sum_{\mu\in H'\setminus\{0\}}\downarrow_H(e^{\gamma+\mu}) \\
    &= e^{\gamma+H} - \frac{1}{p-1}\sum_{\mu\in H'\setminus\{0\}}e^{\gamma+H} \\
    &= 0.
\end{align*}
\end{proof}

\subsection*{The case where $p$ is an odd prime}

In this subsection we assume that $p$ is odd.

First we want to give a necessary condition for $e^\gamma\in\image(\uparrow)$:
\begin{lem}\label{lem:onlyOneIso}
Let $\gamma\in D$. If $e^\gamma\in\image(\uparrow)$ then $\gamma^\bot$ contains at least two isotropic subgroups of order $p$.
\end{lem}
\begin{proof}
We show the contraposition: Assume that $\langle\mu\rangle\subset\gamma^\bot$ is the only isotropic subgroup of order $p$ in $\gamma^\bot$. Let $\ord(\gamma) = n$ and $\q(\gamma) = j/n$. \\
If $\gamma=0$ it is clear because then $\gamma^\bot = D$, so assume $n>1$. First assume that $n=p$ and $j=0\bmod p$. Then $\langle\gamma\rangle = \langle-\gamma\rangle = \langle\mu\rangle$ and so $\downarrow_{\langle\mu\rangle}(e^\gamma-e^{-\gamma}) = 0$. Since for all other subgroups $H$ of order $p$ we have $\gamma,-\gamma\not\in H^\bot$, by definition also $\downarrow_H(e^\gamma-e^{-\gamma}) = 0$. So $e^\gamma-e^{-\gamma}\in\ker(\downarrow)$. Since $\langle e^\gamma-e^{-\gamma},e^\gamma\rangle = 1$,
\begin{equation*}
    e^\gamma\not\in\ker(\downarrow)^\bot = \image(\uparrow).
\end{equation*}
Now we assume that $n>p$ or $(j,p)=1$. We first want to show that $(\gamma+\mu)^\bot$ contains only one isotropic subgroup of order $p$ as well, which then must also be $\langle\mu\rangle$: \\
If $j=0\bmod p$ and $n>p$ we have
\begin{align*}
    \q(n/p\cdot\gamma) = n^2/p^2\cdot j/n = \frac{nj}{p^2} = 0 &\mod1 \quad \text{and} \\
    (n/p\cdot\gamma,\gamma) = 2n/p\cdot\q(\gamma) = 2n/p\cdot j/n = 0&\mod1.
\end{align*}
Hence $\langle\mu\rangle = \langle n/p\cdot\gamma\rangle$ and so $\gamma + \mu = (1+kn/p)\cdot\gamma$ for some $k=1,\hdots,p-1$ and
\begin{equation*}
    ((1+kn/p)\cdot\gamma)^\bot = \gamma^\bot.
\end{equation*}
If $(j,p) = 1$, we can write
\begin{equation*}
    D = \langle\gamma\rangle\oplus \langle\gamma\rangle^\bot.
\end{equation*}
Note that $\ord(\gamma+\mu) = n$ and $\q(\gamma+\mu) = \q(\gamma)$, so we can also write
\begin{equation*}
    D = \langle\gamma+\mu\rangle\oplus\langle\gamma+\mu\rangle^\bot.
\end{equation*}
Since $\langle\gamma\rangle\cong\langle\gamma+\mu\rangle$, also $\langle\gamma\rangle^\bot\cong\langle\gamma+\mu\rangle^\bot$. \\
Now $v$ from Lemma \ref{lem:elementInKernel} is simply
\begin{equation*}
    v = e^\gamma- \frac{1}{p-1}\sum_{k=1}^{p-1}e^{\gamma+k\mu}
\end{equation*}
and $\langle\mu\rangle$ is the only isotropic subgroup of order $p$ orthogonal to any component of $v$. So $\downarrow_{\langle\mu\rangle}(v) = 0$ implies $v\in\ker(\downarrow)$. Since $\langle v,e^\gamma\rangle = 1$
\begin{equation*}
    e^\gamma\not\in\ker(\downarrow)^\bot = \image(\uparrow).
\end{equation*}
\end{proof}
Now we can give a condition that is equivalent to $e^\gamma\in\image(\uparrow)$ in many cases:
\begin{prp}\label{prp:eGammaIff}
Let $\gamma\in D$ be of order $n$ with $\q(\gamma) = j/n$. Assume that if $n>p$, then $j=0\bmod p$. Then $e^\gamma\in\image(\uparrow)$ if and only if $\gamma^\bot$ contains an isotropic subgroup $H\subset D$ isomorphic to $(\Z/p\Z)^2$.
\end{prp}
\begin{proof}
If $\gamma^\bot$ contains an isotropic subgroup $H\subset D$ isomorphic to $(\Z/p\Z)^2$ we can apply Lemma \ref{lem:eGammaInImage}. So assume on the other hand that $\gamma^\bot$ contains no isotropic subgroup $H\subset D$ isomorphic to $(\Z/p\Z)^2$. We construct a $v\in\ker(\downarrow)$ with $\langle v,e^\gamma\rangle = 1$. Then
\begin{equation*}
    e^\gamma\not\in\ker(\downarrow)^\bot = \image(\uparrow).
\end{equation*}
If $\gamma = 0$, then $\gamma^\bot = D$ and Lemma \ref{lem:elementInKernel} yields $v$. So assume that $n>1$. First consider the case $j=0\bmod p$: as in the previous lemma $n/p\cdot\gamma$ is isotropic and orthogonal to $\gamma$. And for any isotropic subgroup $H$ with $\gamma\in H^\bot$, also $n/p\cdot\gamma\in H^\bot$. Hence $\langle n/p\cdot\gamma\rangle$ must be the only isotropic subgroup of order $p$ in $\gamma^\bot$, because if $\gamma^\bot$ contained any other isotropic subgroup of order $p$, say $H$, then $\langle H,n/p\cdot\gamma\rangle$ would contradict the assumption on $\gamma$. But then by Lemma \ref{lem:onlyOneIso} $e^\gamma\not\in\image(\uparrow)$. \\
Now assume $(j,p)=1$ and so $n=p$ by the assumption of the proposition. Since $(\gamma,\beta) = 0$ is equivalent to $(-\gamma,\beta) = 0$, also $-\gamma$ satisfies our assumption. Let $v_1$ and $v_2$ be the elements from Lemma \ref{lem:elementInKernel} for $\gamma$ and $-\gamma$ respectively. We define
\begin{equation*}
    v := v_1 - v_2.
\end{equation*}
For any non-trivial isotropic subgroup $H\subset\gamma^\bot$ we have
\begin{equation*}
    \downarrow_H(v) = \downarrow_H(v_1) - \downarrow_H(v_2) = 0 - 0 = 0.
\end{equation*}
Now we want to show that also $\downarrow_H(v)=0$ when $H\not\subset\gamma^\bot$. By Lemma \ref{lem:iPrimeIsI} it suffices to consider isotropic subgroups of order $p$. So let $\mu$ be any isotropic element of order $p$ with $(\gamma,\mu)\not=0\bmod1$. We show that whenever $(\gamma+\mu_1,\mu) = 0\bmod1$ for some isotropic $\mu_1\in\gamma^\bot$ with $\ord(\mu_1) = p$, then there exists exactly one isotropic $\mu_2\in\gamma^\bot$ with $\ord(\mu_2) = p$ such that
\begin{align*}
    \gamma + \mu_1 &= -\gamma + \mu_2 \mod \langle\mu\rangle\text{, i.e.} \\
    \gamma + \mu_1 + l\cdot\mu &= -\gamma + \mu_2
\end{align*}
for some $l=0,\hdots,p-1$. We can reverse the roles of $\gamma$ and $-\gamma$. This shows that the terms in $\downarrow_{(\mu)}(v_1)$ and $\downarrow_{(\mu)}(v_2)$ cancel each other. So assume that $(\gamma,\mu)\not=0\bmod1$, but $(\gamma+\mu_1,\mu)=0\bmod1$. We need to find a suitable $l$ such that
\begin{equation*}
    \mu_2 := 2\gamma + \mu_1 + l\cdot\mu
\end{equation*}
is isotropic, orthogonal to $\gamma$ and of order $p$. We have
\begin{align*}
    \q(\mu_2) &= \q(2\gamma + \mu_1 + l\cdot\mu) \\
    &= 4\q(\gamma) + 2l(\gamma,\mu) + l(\mu_1,\mu) \\
    &= 4\q(\gamma) + l(\gamma+\mu_1,\mu) + l(\gamma,\mu) \\
    &= 4j/p + l(\gamma,\mu).
\end{align*}
Since $(\gamma,\mu)\not=0\bmod1$, there exists exactly one $l\bmod p$ such that $\mu_2$ is isotropic. With said $l$ we have
\begin{align*}
    (\gamma,\mu_2) &= (\gamma,2\gamma + \mu_1 + l\cdot\mu) \\
    &= 4\q(\gamma) + (\gamma,\mu_1) + l(\gamma,\mu) \\
    &= 4j/p + l(\gamma,\mu) = 0\mod 1.
\end{align*}
Clearly $p\cdot\mu_2=0$. Furthermore, $\mu_2$ cannot be $0$, because then $l\cdot\mu = -(2\gamma + \mu_1)$ and 
\begin{equation*}
    0 = \q(l\mu) = 4\q(\gamma) = 4j/p\mod1,
\end{equation*}
which is a contradiction.
\end{proof}
We will later see an example of a $\gamma$ of order $n>p$ and norm $j/p$ with $(j,p) = 1$ where $\gamma^\bot$ contains no isotropic subgroup isomorphic to $(\Z/p\Z)^2$ but still $e^\gamma\in\image(\uparrow)$. Since for components of level higher than $p$, elements of order $p$ are always isotropic, it is useful to decompose $D = A \oplus B$ where $A$ denotes the sum over the irreducible components of order $p$ and $B$ the sum over the remaining components. So $A=p^{\epsilon n}$ for some $\epsilon=\pm1$ and $n\geq0$. To better understand the isotropic subgroups of $p^{\pm n}$ we need
\begin{prp}
	Let $D=p^{\epsilon n}$ and let $H\subset D$ be a maximal isotropic subgroup of $D$. Then $H\cong(\Z/p\Z)^r$ with
	\[ r = \begin{cases}
		n/2 & \text{if $n$ is even and $\epsilon=\leg{-1}{p}^{n/2}$} \\
		(n-1)/2 & \text{if $n$ is odd} \\
		(n-2)/2 & \text{if $n$ is even and $\epsilon=-\leg{-1}{p}^{n/2}$}
		\end{cases}. \]
\end{prp}
\begin{proof}
	Since $H$ is maximal, the discriminant form $H^\bot/H$ contains no non-trivial isotropic elements. It is well-known and not difficult to prove that the only $p$-adic discriminant forms with this property are $\{0\}$, $p^{\pm1}$ and $p^{-\varepsilon2}$ with $\varepsilon = \leg{-1}{p}$. Because $|H^\bot/H| = |D|/|H|^2$, we find that
	\[ \rk(H^\bot/H) = n-2r. \]
	For $n$ odd, this proves the claim, for even $n$ it follows from $\sign(H^\bot/H) = \sign(D)\bmod8$.
\end{proof}
Now we want to see for which discriminant forms $e^\gamma\in\image(\uparrow)$ for all $\gamma$. First we find those discriminant forms that satisfy the general condition of Corollary \ref{cor:global}
\begin{prp}\label{prp:conditionsHoch3}
Assume that $D$ contains no isotropic subgroup isomorphic to $(\Z/p\Z)^3$. Then $D$ satisfies one of the following conditions:
\begin{enumerate}[(i)]
    \item $D$ has rank two or less.
    \item $D$ has rank three and at least one Jordan component is of level $p$.
    \item $D$ has rank four and $D=p^{-\epsilon2}q_1^{\pm1}q_2^{\pm1}$, where $\epsilon = \leg{-1}{p}$ and $q_1,q_2$ are powers of $p$ and can also be $p$.
    \item $D$ has rank five and $D=p^{-4}q^{\pm1}$, where $q$ is a power of $p$ and can also be $p$.
    \item $D$ has rank six and $D=p^{-\epsilon6}$, where $\epsilon = \leg{-1}{p}$.
\end{enumerate}
\end{prp}
\begin{proof}
We choose a Jordan decomposition of $D$ and write $D = A \oplus B$ where $A$ denotes the sum over the irreducible components of order $p$ and $B$ the sum over the remaining components. Since all elements of order $p$ in $B$ are isotropic and orthogonal to each other, $B$ must have rank less than three. If it has rank two, $p^{\pm n}$ cannot contain any non-trivial isotropic element, so it is equal to one of
\begin{equation*}
    \{0\},\ p^{\pm1},\ p^{-\leg{-1}{p}2}.
\end{equation*}
If $B$ has rank one, $p^{\pm n}$ cannot contain any isotropic subgroup isomorphic to $(\Z/p\Z)^2$, so it is equal to one of
\begin{equation*}
    \{0\},\ p^{\pm1},\ p^{\pm 2}, \ p^{\pm3},\ p^{-4}.
\end{equation*}
If $B$ is trivial, $p^{\pm n}$ cannot contain any isotropic subgroup isomorphic to $(\Z/p\Z)^3$, so it is equal to one of
\begin{equation*}
    \{0\},\ p^{\pm1},\ p^{\pm 2}, \ p^{\pm3},\ p^{\pm4}, \ p^{\pm5},\ p^{-\leg{-1}{p}6}.
\end{equation*}
It is easy to check that the discriminant forms in this list are exactly the ones obtained from the conditions stated in the proposition.
\end{proof}

Finally we check for which discriminant forms appearing in Proposition \ref{prp:conditionsHoch3} indeed $\image(\uparrow)=\C[D]$. 
\begin{thm}\label{thm:mainp}
Let $D$ be a discriminant form of level a power of an odd prime $p$. Then $\image(\uparrow) = \C[D]$ unless $D$ is one of the following:
\begin{enumerate}[(i)]
    \item $D$ has rank two or less.
    \item $D$ has rank three and at least one Jordan component is of level $p$.
    \item $D$ has rank four and $D=p^{-\epsilon2}q_1^{\pm1}q_2^{\pm1}$, where $\epsilon = \leg{-1}{p}$ and $q_1,q_2$ are powers of $p$ and can also be $p$.
    \item $D$ has rank five and is of level $p$.
\end{enumerate}
\end{thm}
\begin{proof}
If $D$ contains an isotropic subgroup isomorphic to $(\Z/p\Z)^3$, then by Corollary \ref{cor:global} $\image(\uparrow) = \C[D]$. Therefore, we consider the remaining discriminant forms, described in Proposition \ref{prp:conditionsHoch3}. We show that in all cases except $p^{-\leg{-1}{p}6}$ and $p^{-4}q^{\pm1}$, there is an element $e^\gamma\not\in\image(\uparrow)$. For $p^{-\leg{-1}{p}6}$ and $p^{-4}q^{\pm1}$, we show that $\image(\uparrow) = \C[D]$. We begin with the latter two cases. \\
So let $\gamma\in p^{-\leg{-1}{p}6}$. We first assume that $\gamma\not=0$. If $\q(\gamma) = 0\bmod1$ we can write
\begin{equation*}
    D = \langle\gamma,\mu\rangle\oplus p^{-4},
\end{equation*}
for some $\mu\in D$ where $\langle\gamma,\mu\rangle \cong p^{\leg{-1}{p}2}$. The component $p^{-4}$ contains a non-trivial isotropic element, say $\beta$. Then $\langle\gamma,\beta\rangle$ is an isotropic subgroup isomorphic to $(\Z/p\Z)^2$ in $\gamma^\bot$ and we can apply Lemma \ref{lem:eGammaInImage}. Now assume that $\q(\gamma) = x/p$ with $\leg{2x}{p} = \epsilon=\pm1$. Then we can write
\begin{equation*}
    D = \langle\gamma\rangle\oplus p^{+4}\oplus p^{-\epsilon\leg{-1}{p}}.
\end{equation*}
The component $p^{+4}$ contains an isotropic subgroup isomorphic to $(\Z/p\Z)^2$ and we can again apply Lemma \ref{lem:eGammaInImage}. We have seen that $p^{-\leg{-1}{p}6}$ contains isotropic subgroups $H$ isomorphic to $(\Z/p\Z)^2$. Obviously $H\subset 0^\bot = D$, hence, also $e^0\in\image(\uparrow)$. \\
Now let $D=p^{-4}q^{\pm1}$ with $q>p$ and write
\begin{equation*}
    D=p^{-4}\oplus\langle\gamma\rangle.
\end{equation*}
Let $\q(\gamma) = j/q$ with $(j,p) = 1$. It is not difficult to see that for any $\beta\in p^{-4}$, there exists a non-trivial isotropic $\mu\in p^{-4}\cap\beta^\bot$. And so for any element of the form $\beta+x\gamma$ with $x=0\bmod p$, the isotropic group $H=\langle\mu,q/p\cdot\gamma\rangle$ is isomorphic to $(\Z/p\Z)^2$ and in $(\beta+x\gamma)^\bot$ because
\[ (q/p\cdot\gamma,\beta+x\gamma) = 2xq/p\cdot\q(\gamma) = 0\mod 1. \]
If $(x,p) = 1$, we can w.l.o.g. assume that $\beta=0$, because one can also decompose $D=p^{-4}\oplus\langle x\gamma+\beta\rangle$. Let $I_p$ be the set of isotropic elements in $p^{-4}$ of order $p$. We define
\begin{align*}
    v &:= \sum_{\mu\in I_p}\uparrow_{\langle\mu\rangle}(e^{\gamma+\langle\mu\rangle}) \\
    w &:= \sumstack{\mu,\beta\in I_p \\ (\mu,\beta)=-2j/p}\uparrow_{\langle(q/p)\cdot\gamma+\beta\rangle}(e^{\gamma+\mu+\langle(q/p)\cdot\gamma+\beta\rangle}) \\
    u_l &:= \sumstack{\mu\in I_p,\beta\in p^{-4} \\ (\mu,\beta)=0 \\ \q(\beta)=-2lj/p}\uparrow_{\langle\mu\rangle}(e^{(1+lq/p)\gamma+\beta+\langle\mu\rangle})
\end{align*}
for $l=1,\hdots,p-1$. Then we have
\begin{align*}
    v = |I_p|\cdot e^\gamma + (p-1)\sum_{\mu\in I_p}e^{\gamma+\mu}.
\end{align*}
The terms in $w$ are of the form $e^{(1+lq/p)\gamma+\alpha}$, where $\alpha=\mu+l\beta$ with $\q(\alpha) = l(\mu,\beta)=-2lj/p$ and $(\alpha,\mu) = -2lj/p$. If we on the other hand start with an $l=1,\hdots,p-1$ and $\alpha\in p^{-4}\setminus\{0\}$ with $\q(\alpha) = -2lj/p$, then we can find a $\mu\in I_p$ with $(\alpha,\mu) = -2lj/p$. Furthermore $\beta:=l^{-1}(\alpha-\mu)\in I_p$ and satisfies $(\mu,\beta)=-2j/p$ and $\alpha=\mu+l\beta$. We denote the number of such $\mu$ for a given $\alpha$ by $a_l$. This number does not depend on the choice of $\alpha$ because all such $\alpha$ are in the same orbit under the automorphism group of $p^{-4}$. For $l=0$ let us denote the number of $\beta\in I_p$ with $(\beta,\mu)=-2j/p\bmod1$ for a given $\mu\in I_p$ by $a_0>0$, which again does not depend on the choice of $\mu$. Then we find
\begin{align*}
    w = \sum_{l=0}^{p-1}a_l\sumstack{\alpha\in p^{-4}\setminus\{0\} \\ \q(\alpha)=-2lj/p}e^{(1+lq/p)\gamma + \alpha}.
\end{align*}
The terms in $u_l$ are of the form $e^{(1+lq/p)\gamma+\alpha}$, where $\alpha=m\mu+\beta$ with $\q(\alpha) = \q(\beta)=-2lj/p$ and $(\alpha,\mu) = 0$. Again for any given $l=1,\hdots,p-1$ and $\alpha\in p^{-4}\setminus\{0\}$ with $\q(\alpha) = -2lj/p$ we find a $\mu\in I_p$ with $(\alpha,\mu) = 0$. Then for every $m=0,\hdots,p-1$ we find that $\beta:=\alpha-m\mu\in p^{-4}$ satisfies $(\mu,\beta)=0$, $\q(\beta)=-2lj/p$ and $\alpha=m\mu+\beta$. We denote the number of such $\mu$ for a given $\alpha$ by $b_l>0$, which again does not depend on the choice of $\alpha$. Then we find
\begin{align*}
    u_l = pb_l\sumstack{\alpha\in p^{-4}\setminus\{0\} \\ \q(\alpha)=-2lj/p}e^{(1+lq/p)\gamma + \alpha}.
\end{align*}
Together we get
\begin{align*}
    \frac{1}{|I_p|}\left[v - \frac{p-1}{a_0}\cdot w + \sum_{l=1}^{p-1}\frac{(p-1)a_l}{a_0pb_l}\cdot u_l\right] = e^\gamma.
\end{align*}
Now we go through the four conditions provided in the theorem and show that in each case we can always find a $\gamma$ for which $e^\gamma$ is not in $\image(\uparrow)$: \\
If $D$ has rank two or less, we can write $D=\langle\beta\rangle\oplus\langle\gamma\rangle$, where $\beta$ and $\gamma$ are either anisotropic or trivial. So $\gamma^\bot = \langle\beta\rangle$ contains at most one isotropic subgroup of order $p$. By Lemma \ref{lem:onlyOneIso} $e^\gamma\not\in\image(\uparrow)$. \\
Now let $D = \langle\mu\rangle\oplus\langle\beta\rangle\oplus\langle\gamma\rangle$, where $\mu,\beta$ and $\gamma$ are all anisotropic and $\q(\mu) = x/p\bmod1$ for some integer $x$ coprime to $p$. Then $\gamma^\bot = \langle\mu\rangle\oplus\langle\beta\rangle$ and $\beta^\bot=\langle\mu\rangle\oplus\langle\gamma\rangle$. If $\ord(\beta) = n > p$, then $\langle n/p\cdot\beta\rangle$ is the only isotropic subgroup of order $p$ in $\gamma^\bot$ and so $e^\gamma\not\in\image(\uparrow)$. If $\ord(\beta) = p$, the only subgroup in $\beta^\bot$ isomorphic to $(\Z/p\Z)^2$ is $\langle\mu\rangle\oplus\langle\ord(\gamma)/p\cdot\gamma\rangle$, which is not isotropic since $\mu$ is not. So by Proposition \ref{prp:eGammaIff} $e^\beta\not\in\image(\uparrow)$. \\
Now let $D=p^{-\leg{-1}{p}2}\oplus\langle\beta\rangle\oplus\langle\gamma\rangle$, where $\beta$ and $\gamma$ are anisotropic. Similar to before, if $\beta$ is of order $n>p$, then $\langle n/p\cdot\beta\rangle$ is the only isotropic subgroup of order $p$ in $\gamma^\bot$ and we can apply Lemma \ref{lem:onlyOneIso}. If $\ord(\beta) = p$, then we can assume that $\gamma$ has order $p$ as well, because otherwise we are in the same situation as before just with the roles of $\beta$ and $\gamma$ reversed. And so $\beta^\bot\cong p^{\pm3}$, which contains no isotropic subgroup isomorphic to $(\Z/p\Z)^2$ and so by Proposition \ref{prp:eGammaIff} $e^\beta\not\in\image(\uparrow)$. \\
Finally let $D = p^{-4}\oplus\langle\gamma\rangle$, with $\gamma$ of order $p$. Then $\gamma^\bot = p^{-4}$, which contains no isotropic subgroup isomorphic to $(\Z/p\Z)^2$ and so again by Proposition \ref{prp:eGammaIff} $e^\gamma\not\in\image(\uparrow)$.
\end{proof}

\subsection*{The case where $p=2$}

Now we consider the prime $p=2$. The situation is similar to $p$ odd, but more complicated. \\
We assume that $D$ is a discriminant form of level a power of $2$. We define a graph $G_D=(V,E)$ with set of vertices $V$ and set of edges $E$ by
\begin{align*}
    V &:= D, \\
    E &:= \{\{\gamma,\beta\}\subset V\mid\mu:=\gamma-\beta,\ord(\mu)=2, \\
    &\qquad\qquad\q(\mu) = (\mu,\gamma) = (\mu,\beta) = 0\bmod1\},
\end{align*}
i.e. there is an edge between $\gamma$ and $\beta$ if and only if $\{\gamma,\beta\}$ is a coset in $H^\bot/H$ for some isotropic subgroup $H$ of order $2$. If for $\gamma_1,\hdots,\gamma_n\in V$ there is an edge between $\gamma_i$ and $\gamma_{i+1}$ for $i=1,\hdots,n-1$ and an edge between $\gamma_n$ and $\gamma_1$, we call $(\gamma_1,\hdots,\gamma_n)$ an \textit{isotropic cycle} of length $n$. Recall that a graph $G=(V,E)$ is called \textit{bipartite} if one can partition $V$ into two disjoint sets $A$ and $B$ such that for no $e\in E$ we have $e\subset A$ or $e\subset B$. We have

\begin{prp}\label{prp:eGammaIff21}
Let $\gamma\in D$. Then the following are equivalent:
\begin{enumerate}[(i)]
    \item $e^\gamma\in\image(\uparrow)$
    \item The connected component of $G_D$ containing $\gamma$ is not bipartite
    \item $G_D$ contains an isotropic cycle containing $\gamma$ that has odd length
\end{enumerate}
\end{prp}
\begin{proof}
We first show (i) implies (ii) by contraposition. Suppose that the connected component $G'=(V',E')$ of $G_D$ containing $\gamma$ was bipartite with decomposition $V' = A\cup B$. We can assume that $\gamma\in A$ and define
\begin{equation*}
    v := \sum_{\beta\in A}e^\beta - \sum_{\beta\in B}e^\beta.
\end{equation*}
Now let $H$ be any isotropic subgroup of order $2$. If $\beta\in H^\bot$ for some $\beta\in A$, then $\beta+H$ is an edge in $G'$ and so the other element in this coset is in $B$. By the same reasoning if $\beta\in B$ the other element of $\beta+H$ is in $A$. Therefore
\begin{align*}
    \downarrow_{H}(v) &= \sum_{\beta\in A}\downarrow_{H}(e^\beta) - \sum_{\beta\in B}\downarrow_{H}(e^\beta) \\
    &= \sum_{\beta\in A}\downarrow_{H}(e^\beta) - \sum_{\beta\in A}\downarrow_{H}(e^\beta) \\
    &= 0.
\end{align*}
By Lemma \ref{lem:iPrimeIsI} this implies that $v\in\ker(\downarrow)$. Since $\langle v,e^\gamma\rangle = 1$, we know
\begin{equation*}
    e^\gamma\not\in\ker(\downarrow)^\bot = \image(\uparrow).
\end{equation*}
Now assume that the connected component $G'=(V',E')$ of $G_D$ containing $\gamma$ is not bipartite. We want to show that $G'$ contains an isotropic cycle containing $\gamma$ that has odd length. It is a well-known fact from graph theory that a graph is bipartite if and only if it contains no cycle of odd length. So let $(\beta_1,\hdots,\beta_n)$ with $n$ odd be an isotropic cycle in $G'$, which must exist by assumption. Since $G'$ is connected there exists a path $(\gamma,\gamma_1,\hdots,\gamma_m,\beta_1)$ in $G'$. But then
\begin{align*}
    (\gamma,\gamma_1,\hdots,\gamma_m,\beta_1,\hdots,\beta_n,\beta_1,\gamma_m,\hdots,\gamma_1)
\end{align*}
is an isotropic cycle containing $\gamma$ of length $1+m+n+1+m = 2m+2+n$, which is odd. \\
Finally we want to show (iii) implies (i) so let $(\gamma_1,\hdots,\gamma_n)$ be an isotropic cycle with $n$ odd and $\gamma_1 = \gamma$ and set $\gamma_{n+1} = \gamma_1$. Let $\mu_i = \gamma_{i+1}-\gamma_i$ for $i=1,\hdots,n$. By definition $\mu_i$ is isotropic and orthogonal to $\gamma_i$, so 
\begin{equation*}
    -\frac{1}{2}\sum_{i=1}^n(-1)^i\uparrow_{\langle\mu_i\rangle}(e^{\gamma_i})
\end{equation*}
is well-defined and in $\image(\uparrow)$. We have
\begin{align*}
    -\frac{1}{2}\sum_{i=1}^n(-1)^i\uparrow_{\langle\mu_i\rangle}(e^{\gamma_i}) &= -\frac{1}{2}\sum_{i=1}^n(-1)^i(e^{\gamma_i} + e^{\gamma_{i}+\mu_i}) \\
    &= -\frac{1}{2}\sum_{i=1}^n(-1)^i(e^{\gamma_i} + e^{\gamma_{i+1}}) \\
    &= \frac{1}{2}(e^{\gamma_1} + e^{\gamma_{n+1}}) = e^{\gamma}.
\end{align*}
\end{proof}
If $\gamma^\bot$ contains an isotropic subgroup isomorphic to $(\Z/2\Z)^2$ then Lemma \ref{lem:eGammaInImage} implies that in this case an isotropic cycle of odd length containing $\gamma$ exists. In fact if
\begin{equation*}
    \{0,\mu_1,\mu_2,\mu_1+\mu_2\}
\end{equation*}
is such a group, then
\begin{equation*}
    (\gamma,\gamma+\mu_1,\gamma+\mu_1+\mu_2) 
\end{equation*}
is an isotropic cycle containing $\gamma$ of length $3$. \\
Lemma \ref{lem:onlyOneIso} also holds for $p=2$:
\begin{lem}\label{lem:onlyOneIso2}
	Let $\gamma\in D$. If $e^\gamma\in\image(\uparrow)$ then $\gamma^\bot$ contains at least two isotropic subgroups of order $2$.
\end{lem}
\begin{proof}
	Suppose that $\mu_1$ is the only isotropic element of order $2$ in $\gamma^\bot$ and that $(\gamma_0,\gamma_1,\gamma_2,\hdots,\gamma_{n-1})$ is an isotropic cycle of length $n$ with $n$ odd and $\gamma_0=\gamma$. Then $\mu_1 = \gamma_1-\gamma_0$ and we set $\mu_i = \gamma_i - \gamma_{i-1}$ for $i=2,\hdots,n-1$ and $\mu_n = \gamma_0-\gamma_n$. Note that $\gamma_i = \gamma+\sum_{j=1}^i\mu_j$, $(\mu_i,\gamma_{i-1}) = (\mu_i,\gamma_{i}) = 0\bmod1$ and $\q(\gamma_i) = \q(\gamma)\bmod1$ for all $i=1,\hdots n$. If $n=3$, then $\{0,\mu_1,\mu_2,\mu_3\}$ is an isotropic subgroup isomorphic to $(\Z/2\Z)^2$ in $\gamma^\bot$. Hence, $n\geq5$. We show that we can construct an isotropic cycle of length $n-2$. Then, recursively we can find an isotropic cycle of length $3$, which is a contradiction. If $\mu_{i}=\mu_{i+1}$ for some $i=1,\hdots,n$, then $(\gamma_0,\hdots,\gamma_{i-1},\gamma_{i+2},\hdots,\gamma_n)$ is an isotropic cycle of length $n-2$. So assume that $\mu_{i}\not=\mu_{i+1}$, in particular $\mu_1\not=\mu_2$. This implies
	\[ (\mu_2,\mu_1) = (\mu_2,\gamma) = 1/2\mod1. \]
	If $\mu_3=\mu_1$, then $\gamma_3 = \gamma+\mu_2$ and so
	\[ (\gamma,\mu_2) = \q(\gamma_3) - \q(\gamma) - \q(\mu_2) = 0\mod1, \]
	which is a contradiction. Therefore, $(\mu_3,\gamma) = 1/2\bmod1$ and so $(\mu_3,\mu_1+\mu_2) = 1/2\bmod1$. This implies that $\mu_1+\mu_2+\mu_3$ is isotropic and in $\gamma^\bot$, i.e. equal to $\mu_1$. But then $(\gamma_0,\gamma_3,\hdots,\gamma_{n-1})$ is an isotropic cycle of length $n-2$.
\end{proof}

As we did for odd primes we want to find those discriminant forms that do not contain an isotropic subgroup isomorphic to $(\Z/2\Z)^3$. Again it is useful to decompose the discriminant form $D$ into those components where all elements of order $2$ are isotropic, and those where this is not the case. The former is the case for Jordan components of type $4_{I\!I}^{\pm n}$ and irreducible components generated by elements of order divided by $8$. The latter therefore consists of components of type $2_{I\!I}^{\pm n}$, $2_{t}^{\pm n}$ and $4_{s}^{\pm m}$. Furthermore note that
\begin{align*}
    2_t^{\epsilon}\cong2_{t+4}^{-\epsilon}
\end{align*}
and $\epsilon e(t/8)$ and $\epsilon\leg{t}{2}$ are invariant under this change of symbols. Therefore, we can always assume that $\epsilon=+1$ and $2_t^{\epsilon n}$ is generated by pairwise orthogonal elements $\mu_1,\hdots,\mu_n$ with $\q(\mu_i) = t_i/4\bmod1$, where $t_i=\pm 1$.

For discriminant forms of level $p$ odd it was easy to see, when they do not contain any isotropic subgroups of large rank. For discriminant forms consisting of components of type $2_{I\!I}^{\pm n}$, $2_{t}^{\pm n}$ and $4_{s}^{\pm m}$ it is more difficult. We will do this in the next three lemmata.


\begin{lem}\label{lem:D1}
Let $\mathcal{D}_1'$ be the set of discriminant forms $D$ such that $D$ is a sum of components of type $2_{I\!I}^{\pm n}$, $2_{t}^{\pm n}$ and $4_{s}^{\pm m}$ and $D$ contains no non-trivial isotropic elements. Then $\mathcal{D}_1'$ consists of the following discriminant forms:
\begin{align*}
    &\{0\}, \ 2_{I\!I}^{-2}, \ 2_{t}^{\pm1}, \\
    &2_{t}^{\pm2}, \ \text{where }t=2\bmod4, \\
    &2_{t}^{\epsilon3}, \ \text{where }\epsilon\leg{t}{2}=-1, \\
    &4_{t}^{\pm1}, \ 2_{t}^{\pm1}4_{s}^{\pm1}.
\end{align*}
\end{lem}
\begin{proof}
Let $D$ be a sum of components of type $2_{I\!I}^{\pm n}$, $2_{t}^{\pm n}$ and $4_{s}^{\pm m}$. First we assume that $D$ is of type $2_{I\!I}^{\pm n}$. By definition $2_{I\!I}^{+2}$ contains a non-trivial isotropic element and if $n\geq4$, we can write $2_{I\!I}^{\pm n}\cong2_{I\!I}^{+2}\oplus2_{I\!I}^{\pm (n-2)}$, so that only $\{0\}$ and $2_{I\!I}^{-2}$ are in $\mathcal{D}_1'$. \\
Now we assume that $D$ has level $4$, i.e. $D$ is of type $2_t^{\pm n}$. It is an easy exercise to see that $2_{t}^{\pm1}$, $2_{t}^{\pm2}$ with $t=2\bmod4$ and $2_{t}^{\epsilon3}$ with $\epsilon\leg{t}{2}=-1$ are the only discriminant forms that do not contain non-trivial isotropic elements. \\
Now we assume that $D$ has level $8$, so $D$ is of the form $4_s^{\pm n}$, $2_{I\!I}^{-2}4_s^{\pm n}$ or $2_t^{\pm m}4_s^{\pm n}$. If $n\geq2$ then it is again easy to see that there is an isotropic element of order $2$. So $n=1$ and $4_s^{\pm1}$ is generated by an element $\gamma$ and $\q(2\gamma) = 1/2\bmod1$. Therefore, the other component must not contain a non-trivial element of norm $0$ or $1/2$. This is the case only for $\{0\}$ and $2_t^{\pm 1}$.
\end{proof}

Now we consider isotropic subgroups isomorphic to $(\Z/2\Z)^2$.

\begin{lem}\label{lem:D2}
Let $\mathcal{D}_2'$ be the set of discriminant forms $D$ such that $D$ is a sum of components of type $2_{I\!I}^{\pm n}$, $2_{t}^{\pm n}$ and $4_{s}^{\pm m}$ and $D$ contains no isotropic subgroup isomorphic to $(\Z/2\Z)^2$. Then $\mathcal{D}_2'$ consists of the following discriminant forms:
\begin{align*}
    &\{0\}, \ 2_{I\!I}^{\pm2}, \ 2_{I\!I}^{-4}, \\
    &2_t^{\pm n}, \ \text{where } n\leq3 \\
    &2_t^{\epsilon 4}, \ \text{where } \epsilon e(t/8) \not= 1 \\
    &2_t^{\epsilon 5}, \ \text{where } \epsilon\leg{t}{2} = -1 \\
    &4_s^{\pm1}, \ 2_{I\!I}^{\pm2}4_s^{\pm1}, \\
    &2_t^{\pm n}4_s^{\pm1}, \ \text{where } n\leq3 \\
    &4_s^{\pm2}, \ 2_t^{\pm 1}4_s^{\pm2}.
\end{align*}
\end{lem}
\begin{proof}
Let $D$ be a sum of components of type $2_{I\!I}^{\pm n}$, $2_{t}^{\pm n}$ and $4_{s}^{\pm m}$. First we assume that $D$ is of type $2_{I\!I}^{\pm n}$. Of course $2_{I\!I}^{\pm 2}\cong(\Z/2\Z)^2$, which contains an anisotropic element and so $2_{I\!I}^{\pm 2}\in\mathcal{D}_2'$. If $n\geq4$ we can write $2_{I\!I}^{\pm n}\cong2_{I\!I}^{+ 2}\oplus2_{I\!I}^{\pm (n-2)}$ and $2_{I\!I}^{\pm (n-2)}$ must be in $\mathcal{D}_1'$, so we only get $2_{I\!I}^{-4}\in\mathcal{D}_2'$. \\
Now we assume that $D$ has level $4$. Then $D$ is of type $2_t^{\epsilon n}$. If $D\in\mathcal{D}_1'$ then also $D\in\mathcal{D}_2'$. Otherwise $D$ contains an isotropic element $\gamma$ of order $2$ and $\langle\gamma\rangle^\bot/\langle\gamma\rangle\cong 2_t^{\epsilon (n-2)}$ or $\langle\gamma\rangle^\bot/\langle\gamma\rangle\cong 2_{I\!I}^{\epsilon' (n-2)}$, with $\epsilon' = \epsilon(-1)^{(t-4)/2}$. Now $D\in\mathcal{D}_2'$ if and only if $\langle\gamma\rangle^\bot/\langle\gamma\rangle\in\mathcal{D}_1'$. This shows that $D\in\mathcal{D}_2'$ when $n\leq3$. For $n=4$ we see that $D\in\mathcal{D}_2
$ if $\epsilon e(t/8) = e(\sign(D)/8) \not=1$ and for $n=5$ if $\epsilon\leg{t}{2} = -1$. \\
Now we assume that $D$ has level $8$ and so $D = 2_{I\!I}^{\pm m}4_s^{\pm n}$ or $D = 2_{t}^{\pm m}4_s^{\pm n}$. If $n=1$ then $4_s^{\pm 1}$ is generated by an element $\gamma$ of order $4$ and $\q(2\gamma) = 1/2\bmod1$. So $2_{I\!I}^{\pm m}$ respectively $2_{t}^{\pm m}$ must not contain any isotropic subgroup isomorphic to $(\Z/2\Z)^2$, but also no distinct $\mu_1,\mu_2$ with $(\mu_1,\mu_2)=0\bmod1$ and $\q(\mu_1)=\q(\mu_2) = 1/2\bmod1$, because in the latter case $\mu_1+2\gamma$ and $\mu_2+2\gamma$ generate an isotropic subgroup isomorphic to $(\Z/2\Z)^2$. We have already seen that only
\begin{align*}
    &\{0\}, \ 2_{I\!I}^{\pm2}, \ 2_{I\!I}^{-4}, \\
    &2_t^{\pm n}, \ \text{where } n\leq3 \\
    &2_t^{\epsilon 4}, \ \text{where } \epsilon e(t/8) \not= 1 \\
    &2_t^{\epsilon 5}, \ \text{where } \epsilon\leg{t}{2} = -1
\end{align*}
contain no isotropic subgroup isomorphic to $(\Z/2\Z)^2$. Going through this finite list of discriminant forms one finds that only $2_{I\!I}^{-4}$, $2_t^{\epsilon 4}$ and $2_t^{\epsilon 5}$ contain elements $\mu_1,\mu_2$ with the above stated norms. \\
Now we consider $n=2$. Then $4_s^{\pm 2}$ is generated by two anisotropic elements $\gamma,\beta$ of order $4$ orthogonal to each other and $\q(2\gamma+2\beta) = 0\bmod1$ and $\q(2\gamma) = \q(2\beta) = 1/2\bmod1$. And so the other component must not contain any non-trivial element $\mu$ of norm $0$ or $1/2$. This is the case only for $\{0\}$ and $2_t^{\pm 1}$. \\
Finally if $n\geq3$, $4_s^{\pm 3}$ is generated by three anisotropic elements $\gamma,\beta,\mu$ of order $4$ all pairwise orthogonal and $\q(\gamma) = s_1/8\bmod1$, $\q(\beta) = s_2/8\bmod1$ and $\q(\mu) = s_3/8\bmod1$. But then $2\gamma+2\beta$ and $2\gamma+2\mu$ generate an isotropic subgroup isomorphic to $(\Z/2\Z)^2$. This concludes the proof.
\end{proof}

Similarly one proves the result for $(\Z/2\Z)^3$. One gets

\begin{lem}\label{lem:D3}
Let $\mathcal{D}_3'$ be the set of discriminant forms $D$ such that $D$ is a sum of components of type $2_{I\!I}^{\pm n}$, $2_{t}^{\pm n}$ and $4_{s}^{\pm m}$ and $D$ contains no isotropic subgroup isomorphic to $(\Z/2\Z)^3$. Then $\mathcal{D}_3'$ consists of the following discriminant forms:
\begin{align*}
    &\{0\}, \ 2_{I\!I}^{\pm2}, \ 2_{I\!I}^{\pm4}, \ 2_{I\!I}^{-6}, \\
    &2_t^{\pm n}, \ \text{where } n\leq5 \\
    &2_t^{\epsilon 6}, \ \text{where } \epsilon e(t/8) \not= 1 \\
    &2_t^{\epsilon 7}, \ \text{where } \epsilon\leg{t}{2} = -1 \\
    &4_s^{\pm1}, \ 2_{I\!I}^{\pm2}4_s^{\pm1}, \ 2_{I\!I}^{\pm4}4_s^{\pm1}, \\
    &2_t^{\pm n}4_s^{\pm1}, \ \text{where } n\leq5 \\
    &4_s^{\pm2}, \ 2_{I\!I}^{\pm 2}4_s^{\pm2} \\
    &2_t^{\pm n}4_s^{\pm2}, \ \text{where } n\leq3 \\
    &4_s^{\pm3}, \ 2_t^{\pm 1}4_s^{\pm3}.
\end{align*}
\end{lem}
Now we can say under which conditions a $2$-adic discriminant form contains no isotropic subgroup isomorphic to $(\Z/2\Z)^3$.

\begin{prp}\label{prp:conditionsHoch32}
Let $D$ be a discriminant form. We choose a Jordan decomposition of $D$ and write $D = A \oplus B$, where $A$ denotes the sum over the components of type $2_{I\!I}^{\pm n}$, $2_{t}^{\pm n}$ and $4_{s}^{\pm m}$ and $B$ the sum over the remaining components. Then $D$ contains no isotropic subgroup isomorphic to $(\Z/2\Z)^3$ if and only if $B$ has rank $r$ with $0\leq r<3$ and $A\in\mathcal{D}_{3-r}'$.
\end{prp}
\begin{proof}
The elements of order $2$ in $B$ generate an isotropic subgroup isomorphic to $(\Z/2\Z)^r$, so $D$ contains no isotropic subgroup isomorphic to $(\Z/2\Z)^3$ if and only if $A$ contains no isotropic subgroup isomorphic to $(\Z/2\Z)^{3-r}$. The previous lemmata prove the proposition.
\end{proof}

Before we can show for which $2$-adic discriminant forms one has $\image(\uparrow)=\C[D]$, we need

\begin{lem}\label{lem:IndependentOfq}
	Let $A$ be a discriminant form isomorphic to a sum of components of type $2_{I\!I}^{\pm n}$, $2_{t}^{\pm n}$ and $4_{s}^{\pm m}$. Let $q,\Tilde{q}\geq8$ be powers of $2$, $t,\Tilde{t}\in\{1,3,5,7\}$ and $\epsilon = \leg{t}{2},\Tilde{\epsilon} = \leg{\Tilde{t}}{2}$. Then $\image(\uparrow) = \C[A\oplus q_t^{\epsilon}]$ if and only if $\image(\uparrow) = \C[A\oplus \Tilde{q}_{\Tilde{t}}^{\Tilde{\epsilon}}]$.
\end{lem}
\begin{proof}
	Suppose that $\image(\uparrow) = \C[A\oplus q_t^{\epsilon}]$ and let $\Tilde{\gamma}\in A\oplus \Tilde{q}_{\Tilde{t}}^{\Tilde{\epsilon}}$ be arbitrary. We assume that $q_t^{\epsilon}$ is generated by an element $\beta$ with $\q(\beta) = \frac{t}{2q}\bmod1$ and $\Tilde{q}_{\Tilde{t}}^{\Tilde{\epsilon}}$ is generated by an element $\Tilde{\beta}$ with $\q(\Tilde{\beta}) = \frac{\Tilde{t}}{2\Tilde{q}}\bmod1$. Write $\Tilde{\gamma} = \alpha + x\Tilde{\beta}$ with $\alpha\in A$ and $x\in\Z$. By assumption, for $\gamma = \alpha + x\beta\in A\oplus q_t^{\epsilon}$ we have $e^{\gamma}\in\image(\uparrow)$. Note that the elements in the connected component of $\gamma$ in $G_{A\oplus q_t^{\epsilon}}$ are of the form $\gamma+\sum_i\mu_i$, where $\mu_i\in A\oplus q_t^{\epsilon}$ is isotropic of order $2$. Similarly the elements in the connected component of $\Tilde{\gamma}$ in $G_{A\oplus \Tilde{q}_{\Tilde{t}}^{\Tilde{\epsilon}}}$ are of the form $\Tilde{\gamma}+\sum_i\Tilde{\mu_i}$, where $\Tilde{\mu}_i\in A\oplus \Tilde{q}_{\Tilde{t}}^{\Tilde{\epsilon}}$ is isotropic of order $2$. The map
	\[ \mu = \mu' + (q/2)\beta\mapsto\Tilde{\mu} = \mu' + (\Tilde{q}/2)\Tilde{\beta} \]
	defines a bijection from the isotropic elements of order $2$ in $A\oplus q_t^{\epsilon}$ to those in $A\oplus \Tilde{q}_{\Tilde{t}}^{\Tilde{\epsilon}}$. For all isotropic $\mu_1,\mu_2\in A\oplus q_t^{\epsilon}$ of order $2$ we have
	\begin{align*}
		(\Tilde{\gamma},\Tilde{\mu_1}) &= (\alpha,\mu'_1) + x(\Tilde{q}/2)\frac{\Tilde{t}}{\Tilde{q}}\mod1 \\
		&=  (\alpha,\mu'_1) + x\frac{\Tilde{t}}{2}\mod1 \\
		&= (\alpha,\mu'_1) + x(q/2)\frac{t}{q}\mod1 \\
		&= (\gamma,\mu_1)\mod1
	\end{align*}
	and
	\begin{align*}
		(\Tilde{\mu_1},\Tilde{\mu_2}) &= (\mu'_1,\mu'_2) + (\Tilde{q}/2)(\Tilde{q}/2)\frac{\Tilde{t}}{\Tilde{q}}\mod1 \\
		&= (\mu'_1,\mu'_2)\mod1 \\
		&= (\mu_1,\mu_2)\mod1.
	\end{align*}
	This shows that the connected component of $\gamma$ and the connected component of $\Tilde{\gamma}$ are isomorphic graphs and hence $e^{\Tilde{\gamma}}\in\image(\uparrow)$ as well. Reversing the roles of $A\oplus q_t^{\epsilon}$ and $A\oplus \Tilde{q}_{\Tilde{t}}^{\Tilde{\epsilon}}$ proves the lemma.
\end{proof}

As for the $p$-adic case we refine the result of Proposition \ref{prp:conditionsHoch32} to get a precise statement when $\image(\uparrow)=\C[D]$.

\begin{thm}\label{thm:main2}
Let $\mathcal{D}_1=\mathcal{D}_1'$, i.e. the set of discriminant forms consisting of
\begin{align*}
	&\{0\}, \ 2_{I\!I}^{-2}, \ 2_{t}^{\pm1}, \\
	&2_{t}^{\pm2}, \ \text{where }t=2\bmod4, \\
	&2_{t}^{\epsilon3}, \ \text{where }\epsilon\leg{t}{2}=-1, \\
	&4_{t}^{\pm1}, \ 2_{t}^{\pm1}4_{s}^{\pm1},
\end{align*}
let $\mathcal{D}_2$ be the set of discriminant forms consisting of
\begin{align*}
    &\{0\}, \ 2_{I\!I}^{\pm2}, \\
    &2_t^{\pm n}, \ \text{where } n\leq3 \\
    &2_t^{\epsilon 4}, \ \text{where } \epsilon e(t/8) \not= 1 \\
    &4_s^{\pm1}, \ 2_{I\!I}^{\pm2}4_s^{\pm1}, \\
    &2_t^{\pm n}4_s^{\pm1}, \ \text{where } n\leq3 \\
    &4_s^{\pm2}, \ 2_t^{\pm 1}4_s^{\pm2}
\end{align*}
and let $\mathcal{D}_3$ be the set of discriminant forms consisting of
\begin{align*}
    &\{0\}, \ 2_{I\!I}^{\pm2}, \ 2_{I\!I}^{\pm4}, \\
    &2_t^{\pm n}, \ \text{where } n\leq5 \\
    &2_t^{\pm6}, \ \text{where } t = 2\bmod4 \\
    &4_s^{\pm1}, \ 2_{I\!I}^{\pm2}4_s^{\pm1}, \ 2_{I\!I}^{\pm4}4_s^{\pm1}, \\
    &2_t^{\pm n}4_s^{\pm1}, \ \text{where } n\leq5 \\
    &4_s^{\pm2}, \ 2_{I\!I}^{\pm 2}4_s^{\pm2} \\
    &2_t^{\pm n}4_s^{\pm2}, \ \text{where } n\leq3 \\
    &4_s^{\pm3}, \ 2_t^{\pm 1}4_s^{\pm3}.
\end{align*}
Let $D$ be a discriminant form of level a power of $2$. We choose a Jordan decomposition of $D$ and write $D = A \oplus B$, where $A$ denotes the sum over the components of type $2_{I\!I}^{\pm n}$, $2_{t}^{\pm n}$ and $4_{s}^{\pm m}$ and $B$ the sum over the remaining components. Then $\image(\uparrow) = \C[D]$ unless $B$ has rank $r$ with $0\leq r<3$ and $A\in\mathcal{D}_{3-r}$.
\end{thm}
\begin{proof}
If $D$ contains an isotropic subgroup isomorphic to $(\Z/2\Z)^3$, then by Corollary \ref{cor:global} $\image(\uparrow) = \C[D]$. Therefore, we consider the discriminant forms described in Proposition \ref{prp:conditionsHoch32}. We show that for those discriminant forms described in Theorem \ref{thm:main2}, there is an element $e^\gamma\not\in\image(\uparrow)$. For those discriminant forms not in Theorem \ref{thm:main2}, we show that $\image(\uparrow) = \C[D]$. \\
First we assume that $r=2$ and that $A\in\mathcal{D}_1'$. If $B=q_{I\!I}^{\pm 2}$ for $q$ some power of $2$ and $\gamma\in B$ of order $q$, then $\gamma^\bot=\langle\beta\rangle\oplus A$, where $\beta\in B$ is some element of order $q$. Therefore, the only isotropic element of order two orthogonal to $\gamma$ is $(q/2)\cdot\beta$. If on the other hand $B$ is the sum of odd $2$-adic components and generated by two anisotropic elements $\gamma$ and $\beta$ orthogonal to each other, then $\gamma^\bot=\langle\beta\rangle\oplus A$ and so the only isotropic element of order two orthogonal to $\gamma$ is $(\ord(\beta)/2)\cdot\beta$. In both cases Lemma \ref{lem:onlyOneIso2} implies $e^\gamma\not\in\image(\uparrow)$. Hence, $\mathcal{D}_1=\mathcal{D}_1'$. \\
Now we assume that $r\leq1$. If $r=1$, because of Lemma \ref{lem:IndependentOfq}, we can assume that $B = 8_1^{+1}$. Therefore, we only need to check for a finite number of discriminant forms $D$ whether $\image(\uparrow) = \C[D]$, i.e. whether $G_D$ contains no bipartite component. This can be done quickly by a computer. This was done using the computational algebra system Magma \cite{BCP}.
\end{proof}

\section{Main theorem}\label{sec:MainTheorem}

Now we can finally prove the main theorem. We first define what it means for a discriminant form to be of small type.

First let $D$ be a discriminant form of level a power of $p$ for some prime $p$. If $p$ is odd then we say that $D$ is of \textit{small type} if one of the following conditions holds:
\begin{enumerate}[(i)]
	\item $D$ has rank two or less.
	\item $D$ has rank three and at least one Jordan component is of level $p$.
	\item $D$ has rank four and is of type $p^{-\epsilon2}q_1^{\pm1}q_2^{\pm1}$, where $\epsilon = \leg{-1}{p}$ and $q_1,q_2$ are powers of $p$ and can also be $p$.
	\item $D$ has rank five and is of level $p$.
\end{enumerate}
If $p=2$ we choose a Jordan decomposition of $D$ and write $D = A \oplus B$, where $A$ denotes the sum over the components of type $2_{I\!I}^{\pm n}$, $2_{t}^{\pm n}$ and $4_{s}^{\pm m}$ and $B$ the sum over the remaining components. Then we say that $D$ is of small type if $B$ has rank $r$ with $0\leq r<3$ and $A\in\mathcal{D}_{3-r}$, where $\mathcal{D}_1$, $\mathcal{D}_2$ and $\mathcal{D}_3$ are as defined in Theorem \ref{thm:main2}. Recall that the components $2_{I\!I}^{\pm n}$, $2_{t}^{\pm n}$ and $4_{s}^{\pm m}$ are exactly those containing anisotropic elements of order $2$. \\	
Now let $D$ be a discriminant form of arbitrary level $N=\prod_{p|N}p^{\nu_p}$. We say that $D$ is of small type if for all $p\mid N$ the $p$-subgroups $D_{p^{\nu_p}}$ of $D$ are of small type.

\medskip
We remark that any discriminant form of rank $\geq7$ is not of small type and any discriminant form of odd level and rank $\geq6$ is not of small type. We get
\begin{thm}\label{thm:mainthm}
	\mainthm
\end{thm}
\begin{proof}
	If $D$ is of small type then $\image(\uparrow) \subsetneq \C[D]$ by the previous two subsections. From Theorem \ref{thm:MainThmNeedsAllCD} we know that in this case there exist modular forms which are not linear combinations of isotropically lifted modular forms. To show that for a discriminant form not of small type we can write all modular form as linear combinations of modular forms lifted from discriminant forms of small type we use induction on $|D|$. If $D$ is trivial then $D$ is of small type and there is nothing to prove. Now assume that $|D|>1$ and that the assertion holds for discriminant forms of order smaller than $|D|$. If $D$ is of small type then, again, there is nothing to prove. Otherwise we have seen in the previous two subsections that $\image(\uparrow) = \C[D]$. Again using Theorem \ref{thm:MainThmNeedsAllCD} we know that all modular forms for $D$ for all weights $\ka\in\frac{1}{2}\Z$ are linear combinations of isotropically lifted modular forms. Using the induction hypothesis and the fact that the isotropic lifts are transitive we get the result.
\end{proof}

\section{Acknowledgements}

The author acknowledges support by Deutsche Forschungsgemeinschaft through the Collaborative Research Centre TRR 326 \textit{Geometry and Arithmetic of Uniformized Structures}, project number 444845124.

\medskip

Furthermore, I would like to thank my advisor N.\ Scheithauer as well as M.\ Dittmann for stimulating discussions and the anonymous referee for their helpful comments.

\medskip

Declarations of interest: none


\end{document}